\documentclass[british,english]{amsart}

\usepackage[OT1]{fontenc}
\usepackage[utf8]{inputenc}
\usepackage{babel}
\usepackage{amstext}
\usepackage{amsthm}
\usepackage{amssymb}
\usepackage{geometry}
\usepackage{setspace}
\usepackage[pdfusetitle,
 bookmarks=true,bookmarksnumbered=false,bookmarksopen=false,
 breaklinks=false,pdfborder={0 0 1},backref=false,colorlinks=false]
 {hyperref}

\makeatletter
\numberwithin{equation}{section}
\numberwithin{figure}{section}
\theoremstyle{plain}
\newtheorem{thm}{\protect\theoremname}[section]
\theoremstyle{plain}
\newtheorem{prop}[thm]{\protect\propositionname}
\theoremstyle{remark}
\newtheorem{notation}[thm]{\protect\notationname}
\theoremstyle{remark}
\newtheorem{rem}[thm]{\protect\remarkname}
\theoremstyle{plain}
\newtheorem{lem}[thm]{\protect\lemmaname}

\usepackage{hyperref}
\hypersetup{colorlinks=true, allcolors=blue}
\usepackage{tikz-cd}
\usepackage[utf8]{inputenc}
\usepackage[T1]{fontenc}
\usepackage{lmodern}
\usepackage{microtype}
\usepackage{frenchmath}
\renewcommand{\mathbb}[1]{\mathbf{#1}}

\numberwithin{equation}{subsection}

\address{Harish-Chandra Research Institute, A CI of Homi Bhabha National
Institute, Chhatnag Road, Jhusi, Prayagraj - 211019, India}
\email{arpan3141@gmail.com}
\thanks{The author is supported by Ph.D. scholarship from Harish-Chandra Research Institute, A CI of Homi Bhabha National
Institute}

\makeatother

\addto\captionsbritish{\renewcommand{\lemmaname}{Lemma}}
\addto\captionsbritish{\renewcommand{\notationname}{Notation}}
\addto\captionsbritish{\renewcommand{\propositionname}{Proposition}}
\addto\captionsbritish{\renewcommand{\remarkname}{Remark}}
\addto\captionsbritish{\renewcommand{\theoremname}{Theorem}}
\addto\captionsenglish{\renewcommand{\lemmaname}{Lemma}}
\addto\captionsenglish{\renewcommand{\notationname}{Notation}}
\addto\captionsenglish{\renewcommand{\propositionname}{Proposition}}
\addto\captionsenglish{\renewcommand{\remarkname}{Remark}}
\addto\captionsenglish{\renewcommand{\theoremname}{Theorem}}
\providecommand{\lemmaname}{Lemma}
\providecommand{\notationname}{Notation}
\providecommand{\propositionname}{Proposition}
\providecommand{\remarkname}{Remark}
\providecommand{\theoremname}{Theorem}

\begin{document}
\selectlanguage{british}%
\global\long\def\ffpbar{\bar{\mathbb{F}}_{p}}%

\global\long\def\ffq{\mathbb{F}_{q}}%

\global\long\def\calO{\mathcal{O}_{F}}%

\global\long\def\frakp{\mathfrak{p}_{F}}%

\global\long\def\bfi{\mathbf{i}}%

\global\long\def\bfj{\mathbf{j}}%

\global\long\def\bfm{\mathbf{m}}%

\global\long\def\bfn{\mathbf{n}}%

\global\long\def\bfk{\mathbf{k}}%

\global\long\def\bfT{\mathbf{T}}%

\global\long\def\SL{\operatorname{SL}}%

\global\long\def\ind{\mathrm{ind}}%

\global\long\def\Sym{\operatorname{Sym}}%

\global\long\def\Hom{\operatorname{Hom}}%

\global\long\def\ker{\operatorname{Ker}}%

\global\long\def\img{\operatorname{Im}}%

\global\long\def\End{\operatorname{End}}%

\global\long\def\sig{\sigma_{\vec{r}}}%

\title{\selectlanguage{british}%
On freeness of compactly induced mod-$p$ representations of ${\rm SL}_{2}(F)$}
\author{\selectlanguage{british}%
Arpan Das}
\begin{abstract}
Let\foreignlanguage{english}{ $p$ be a prime, and $F$ a non-archimedean
local field with residue characteristic $p$ and ring of integers
$\mathcal{O}_{F}$. Set $G_{S}:={\rm SL}_{2}(F)$ and $K_{0}:={\rm SL}_{2}(\mathcal{O}_{F})$
. For a smooth irreducible $\ffpbar$-representation $\sigma$ of
$K_{0}$, we study the structure of the compact induction ${\rm ind}_{K_{0}}^{G_{S}}(\sigma)$
as a left module over the standard spherical Hecke algebra ${\rm End}_{G_{S}}\left({\rm ind}_{K_{0}}^{G_{S}}(\sigma)\right)$.
We prove that it is free and of infinite rank.}
\end{abstract}

\maketitle
\selectlanguage{english}%

\section{Introduction}

The study of smooth mod-$p$ representations of $p$-adic groups has
been a central theme in the mod-$p$ and $p$-adic Langlands programs.
Following the pioneering work of Barthel-Livné \cite{Barthel-1} on
${\rm GL}_{2}$, much effort has gone into understanding the structure
and classification of such representations, particularly the supersingular
ones. Among the four types -- characters, principal series, Steinberg,
and supersingular -- much is known about the first three, but supersingular
representations remain more mysterious in general. For ${\rm GL}_{2}(\mathbb{Q}_{p})$,
Breuil \cite{BreuilGL2Qp1} achieved a complete classification by
constructing explicit models, but for general local fields and for
groups of higher ranks, the picture remains incomplete and technically
subtle. However, the analogue of Barthel--Livné's classification
for the group ${\rm SL}_{2}(F)$ in the mod-$p$ setting was subsequently
established in \cite{Abdellatif-1,ChengSL2}. Also, Herzig extended
this classification for $p$-adic ${\rm GL}_{n}$ in his pioneering
work \cite{Herzig_GLn,Herzig_satake}. More recently, this classification
has been further generalized to connected reductive $p$-adic groups
through the far-reaching work of Abe-Henniart-Herzig-Vign\'eras (see
\cite{Abe_split,Abe-Henniart-Herzig-Vigneras}). 

In all of the above works the structure of the compactly induced representations
and their pro-$p$-Iwahori invariants plays a crucial role. For example,
if $G$ is split reductive $p$-adic group and $\sigma$ is a weight
of some hyperspecial maximal compact subgroup $K$ then ${\rm ind}_{K}^{G}(\sigma)$
is torsion free as a module over $\mathcal{H}(G,K,\sigma):={\rm End}_{G}({\rm ind}_{K}^{G}(\sigma))$
(see \cite[Corollary 6.5]{Herzig_GLn}). When $G$ has semisimple
rank $1$ and is split or quasi-split, say ${\rm GL}_{2},{\rm SL}_{2}$
or the unramified $U(2,1)$, finer structural results are known. For
instance, in the above three examples the explicit right action of
the pro-$p$-Iwahori Hecke algebra on the pro-$p$-Iwahori invariants
of the compactly induced representations can be computed (see \cite{Barthel-1}
for ${\rm GL}_{2}$, \cite{Das24} for ${\rm SL}_{2}$, and \cite{Xu_hecke_eigenvalues}
for $U(2,1)$). 

In the present paper we consider the $p$-adic group $\SL_{2}$ and
we show that the compactly induced representations as left modules
over their spherical Hecke algebras are free and of infinite rank,
thus refining \cite[Corollary 6.5]{Herzig_GLn} in this special case.
Our main theorem is as follows.
\begin{thm}[Theorem \ref{thm:freeness theorem}]
 --- The compactly induced representation $\ind_{K_{0}}^{G_{S}}(\sig)$
is a free module of infinite rank over the spherical Hecke algebra
$\mathcal{H}(G_{S},K_{0},\sig)$.
\end{thm}

Our result complements the analogous freeness results previously established
for ${\rm GL}_{2}$ (see \cite[Section 5]{Barthel-1}) and for the
unramified unitary group $U(2,1)$ (see \cite{Xu_freeness}). There
are, however, two notable distinctions. First, we isolate and formalize
the key combinatorial--linear algebraic argument that underlies both
proofs in \cite{Barthel-1,Xu_freeness}; this is stated abstractly
as Lemma \ref{lem: technical lemma}. This conceptual clarification
separates the structural core of the argument from the group-specific
computations, making the mechanism of freeness more transparent and
potentially applicable to other $p$-adic groups. Second, we give
an explicit proof of the non-vanishing of the standard Hecke operator
on the subspace of functions supported on the central vertex of the
Bruhat--Tits tree of ${\rm SL}_{2}$; this is Lemma \ref{lem:tau injective on C_0}.
Although this fact is expected on general grounds, its verification
requires a careful analysis of the Hecke action and explicit realization
of mod $p$ irreducibles of the finite group $\SL_{2}(\ffq)$ for
a $p$-power $q$. This, to our knowledge, has not appeared explicitly
in the literature. It was used implicitly, for example, in the proof
of the main theorem of \cite{Xu_freeness}. 

Finally, we note that Lemma \ref{lem:tau(f_n)}, together with \cite[Propositions 4.9 and 4.10]{Das24},
provides a detailed description of the explicit Hecke actions on the
pro-$p$-Iwahori invariants of compactly induced representations of
$\SL_{2}$ in the mod $p$ setting. The significance of this analysis
is underscored both by the main theorem of the present paper and by
the existence of Hecke eigenvalues for $\SL_{2}$ (see \cite[Proposition 4.14]{Das24}),
analogous to the ${\rm GL}_{2}$ case \cite[Proposition 32]{Barthel-1};
see also \cite[Question 8]{abe2017questionsmodprepresentations}. 

\subsection*{Outline of the proof of the main theorem}

-- The proof proceeds in two stages. 
\begin{enumerate}
\item In section \ref{subsec:A-technical-result} we prove a general combinatorial--linear
algebra assertion : if a vector space $V$ over any field admits a
graded decomposition $V=\bigoplus_{k\geq0}C_{k}$ and there is a linear
operator $T$ satisfying certain \emph{injectivity}, \emph{triangularity},
and \emph{filtration compatibity }conditions, then we can choose subsets
$A_{k}\subset C_{k}$ inductively so that the family $\{T^{i}(A_{j})\,|\,0\leq i,j\leq n,\,\,i+j\leq n\}$
of sets are mutually disjoint and their union forms a basis of the
truncated sum $B_{n}:=\bigoplus_{k\leq n}C_{k}$. \\
This yields an explicit combinatorial construction of a free $T$-module
basis $\bigsqcup_{k\geq0}A_{k}$ of $V$.
\item In section \ref{sec: main section} we apply this framework to $V={\rm ind}_{K_{0}}^{G_{S}}(\sigma_{\vec{r}})$
and $T=\tau$ . The relevant subspaces $C_{k}$ are obtained from
the Cartan-Iwahori decomposition 
\[
G_{S}=\bigsqcup_{n\in\mathbb{Z}}K_{0}\alpha_{0}^{n}I_{S}(1).
\]
See the next section for precise meaning of the notations. Then we
use the explicit formula for the action of $\tau$ on a standard function
(eq. \ref{eq:action of tau on standard function}) to find the possible
support points of the image of an element of $C_{k}$ in the tree
of $\SL_{2}$. This yields the inclusions
\[
\tau(C_{0})\subset C_{1}\qquad\text{and}\qquad\tau(C_{k})\subset C_{k-1}\oplus C_{k}\oplus C_{k+1}\quad\text{for }k\geq1.
\]
The injectivity of $\tau:C_{0}\to C_{1}$ is proved in the technical
Lemma \ref{lem:tau injective on C_0}. Finally, we compute the explicit
action of $\tau$ on the pro-$p$-Iwahori invariants and use this
computation to establish the \emph{filtration compatibility} condition
of the abstract setup. Thus we have the desired freeness.
\end{enumerate}

\section{Preliminaries}\label{sec:Preliminaries}

\subsection{General notions }

-- We take $p$ to be a prime throughout, and $\bar{\mathbb{F}}_{p}$
a fixed algebraic closure of the finite field $\mathbb{F}_{p}$ with
$p$ elements. All representations, unless otherwise mentioned, are
considered over $\bar{\mathbb{F}}_{p}$. We recall some generalities
on the abstract representation theory of locally profinite groups.
We let $G$ be any locally profinite group, and $H$ some closed subgroup.
A representation $\pi$ of $G$ is called \textit{smooth} if every
vector $v\in\pi$ is fixed by some compact open subgroup of $G$.
Let $\sigma$ be a smooth representation of $H$. We consider the
following space of functions : 
\[
\text{IND}_{H}^{G}(\sigma):=\{f:G\to\sigma\,|\,f(hg)=\sigma(h)(f(g)),\,\forall g\in G,h\in H\}.
\]
Then, $G$ acts on $\text{IND}_{H}^{G}(\sigma)$ via $(g\cdot f)(g^{\prime}):=f(g^{\prime}g)$.
The \emph{smooth part} of $\text{IND}_{H}^{G}(\sigma)$, that is,
vectors that have open stabilizers, is denoted by $\text{Ind}_{H}^{G}(\sigma)$,
and this subrepresentation is called the \textit{smooth induction}
of $\sigma$. The subrepresentation of $\text{Ind}_{H}^{G}(\sigma)$
consisting of functions $f$ such that the image of its support $\text{Supp}(f)$
inside $H\backslash G$ is compact (equivalently, finite, whenever
$H$ is also open) is denoted by $\text{c-Ind}_{H}^{G}(\sigma)$ or
$\text{ind}_{H}^{G}(\sigma)$, and is called the \textit{compact induction}
of $\sigma$.

In practice, whenever we use compact induction the subgroup $H$ is
typically considered to be open as well. So, for the remaining part
of this subsection we take $H$ to be an open subgroup of $G$. Then,
by virtue of the $H$-linearity, the support of any $f\in\text{ind}_{H}^{G}(\sigma)$
can be written as a finite disjoint union of right $H$-cosets. We
define some standard functions in $\text{ind}_{H}^{G}(\sigma)$. For
$g\in G$ and $v\in\sigma$ we define : 
\[
[g,v](x):=\begin{cases}
\sigma(xg)(v) & \text{if}\,x\in Hg^{-1}\\
0 & \text{otherwise}
\end{cases}.
\]
It can be checked that $g\cdot[g^{\prime},v]=[gg^{\prime},v]$ and
$[gh,v]=[g,\sigma(h)(v)]$ for every $g,g^{\prime}\in G$ and $h\in H$.
Also, any $f\in\text{ind}_{H}^{G}(\sigma)$ can be written as 
\[
f=\sum\limits_{Hg\in\text{Supp(\ensuremath{f})}}[g^{-1},f(g)].
\]

\subsection{Some standard notations related to $\protect\SL_{2}(F)$}

-- Let $F$ be a non-archimedean local field of residue characteristic
$p$. We denote its valuation ring by $\mathcal{O}_{F}$ and its valuation
ideal by $\mathfrak{p}_{F}$ and we fix a \emph{uniformizer }$\varpi_{F}$
that generates $\mathfrak{p}_{F}$. The residue field $k_{F}:=\mathcal{O}_{F}/\mathfrak{p}_{F}$
is a finite field of cardinality $q$ which is a power of $p$. We
set $G_{S}:=\SL_{2}(F)$ throughout this article. It is known that
$G_{S}$ has two maximal compact open subgroups $K_{0}:=\SL_{2}(\mathcal{O}_{F})$
and $K_{1}=\alpha K_{0}\alpha^{-1}$ where $\alpha:={\rm diag}(1,\varpi_{F})\in{\rm GL}_{2}(F)$.
Let $I_{S}(1)$ denote the \emph{pro-$p$-Iwahori subgroup} of $G_{S}$.
We have 
\[
I_{S}(1):=\begin{pmatrix}1+\mathfrak{p}_{F} & \mathcal{O}_{F}\\
\mathfrak{p}_{F} & 1+\mathfrak{p}_{F}
\end{pmatrix}\cap K_{0}.
\]
We finally let $U_{S}(\mathfrak{p}_{F}^{n})\text{ (resp.}\,\bar{U}(\mathfrak{p}_{F}^{n}))$
denote the upper triangular (resp. lower triangular) matrices in $G_{S}$
with top right (resp. bottom left) entry in the fractional ideal $\mathfrak{p}_{F}^{n}$
for $n\in\mathbb{Z}$.

\subsection{Generators of the weights of ${\rm SL}_{2}(\mathcal{O}_{F})$}

-- Let $k_{F}=\ffq$ and $q=p^{n}$. Then it is well known that the
irreducible $\bar{\mathbb{F}}_{p}$-representations of ${\rm SL_{2}}(\ffq)$
(or equivalently, smooth irreducible $\bar{\mathbb{F}}_{p}$-representations
of ${\rm SL}_{2}(\calO)$) are precisely of the form 
\[
{\rm Sym}^{\vec{r}}\bar{\mathbb{F}}_{p}^{2}:={\rm Sym}^{r_{0}}\bar{\mathbb{F}}_{p}^{2}\otimes\cdots\otimes{\rm Sym}^{r_{n-1}}\bar{\mathbb{F}}_{p}^{2},
\]
where $\vec{r}=(r_{0},\dots,r_{n-1})\in\{0,\dots,p-1\}^{n}$ and ${\rm SL}_{2}(\mathbb{F}_{q})$
acts on ${\rm Sym}^{r_{i}}\bar{\mathbb{F}}_{p}^{2}:=\bigoplus_{l=0}^{r_{i}}\ffpbar X^{r_{i}-l}Y^{l}$
via : 
\[
\begin{pmatrix}\begin{array}{cc}
a & b\\
c & d
\end{array}\end{pmatrix}\cdot(X^{r_{i}-l}Y^{l}):=(a^{p^{i}}X+c^{p^{i}}Y)^{r_{i}-l}(b^{p^{i}}X+c^{p^{i}}Y)^{l}.
\]
Consider the element $X^{\vec{r}}:=X^{r_{0}}\otimes X^{r_{1}}\otimes\cdots\otimes X^{r_{n-1}}\in{\rm Sym}^{\vec{r}}\bar{\mathbb{F}}_{p}^{2}$.
The following result is well known. We will however give a proof to
keep things self-contained, and also because some arguments in the
proof will be used later.
\begin{prop}
\label{prop:spanning set for weights} --- The representation ${\rm Sym}^{\vec{r}}\bar{\mathbb{F}}_{p}^{2}$
is generated by $X^{\vec{r}}$ as an $\bar{\mathbb{F}}_{p}[\bar{U}]$-module,
and the line $\bar{\mathbb{F}}_{p}\cdot X^{\vec{r}}$ is the unique
$U$-invariant line in ${\rm Sym}^{\vec{r}}\bar{\mathbb{F}}_{p}^{2}$.
Here, $U$ (resp. $\bar{U}$) denotes the upper (resp. lower) unipotent
subgroup of ${\rm SL}_{2}(\mathbb{F}_{q})$.

Consequently, the set $\{\bar{u}\cdot X^{\vec{r}}\,|\,\bar{u}\in\bar{U}_{S}(\calO)\}$
spans ${\rm Sym}^{\vec{r}}\bar{\mathbb{F}}_{p}^{2}$, and the line
$\ffpbar\cdot X^{\vec{r}}$ is the unique $I_{S}(1)$-invariant line
in $\sigma_{\vec{r}}$. 
\end{prop}

\begin{proof}
Write $r=r_{0}+r_{1}p+\cdots+r_{n-1}p^{n-1}$; then $0\leq r<q=p^{n}$.
Consider $\Sym^{\vec{r}}\ffpbar^{2}$ as an $\SL_{2}(\ffq)$-subrepresentation
of $\Sym^{r}\ffpbar^{2}$ via the map $v_{0}\otimes v_{1}\otimes\cdots\otimes v_{n-1}\mapsto v_{0}v_{1}^{p}\cdots v_{n-1}^{p^{n-1}}$.
Then, a basis of $\Sym^{\vec{r}}\ffpbar^{2}$ is given by the family
of monomials $X^{\sum_{j=0}^{n-1}i_{j}p^{j}}Y^{r-(\sum_{j=0}^{n-1}i_{j}p^{j})},$
with $0\leq i_{j}\leq r_{j}$ for all $j\in\{0,\dots,n-1\}.$ 

Now, we recall that for $i=i_{0}+i_{1}p+\cdots+i_{n-1}p^{n-1}\leq r=r_{0}+r_{1}p+\cdots+r_{n-1}p^{n-1},$
the binomial coefficient $\binom{r}{i}$ in $\ffpbar$ is given by
\[
\binom{r}{i}=\begin{cases}
\binom{r_{0}}{i_{0}}\binom{r_{1}}{i_{1}}\cdots\binom{r_{n-1}}{i_{n-1}} & \,{\rm if}\,i_{l}\leq r_{l}\,{\rm for}\,{\rm all\,}l\\
0 & {\rm \,otherwise}
\end{cases}.
\]
This simply follows by writing $(X+1)^{r}\in\ffpbar[X]$ as $(X+1)^{r_{0}}(X^{p}+1)^{r_{1}}\cdots(X^{p^{n-1}}+1)^{r_{n-1}}$
and then expanding and comparing the binomial coefficients.

Consequently, the vector $X^{i}Y^{r-i}$ (for $0\leq i\leq r$) is
in the above basis if and only if $\binom{r}{i}\neq0.$ Therefore,
for any $\lambda\in\ffq$ we take $\bar{u}(\lambda)\in\bar{U}$, and
then we have 
\[
\bar{u}(\lambda)\cdot X^{r}=\sum_{i=0}^{r}\binom{r}{i}\lambda^{r-i}X^{i}Y^{r-i}\in\ffpbar[\bar{U}]\cdot X^{r},
\]
and the non-zero terms of the above sum correspond to those $i$ for
which $\binom{r}{i}\neq0$. Since $\lambda$ has $q$ many possible
values and $r<q,$ an elementary argument in linear algebra shows
that for each $i\leq r$ with $\binom{r}{i}\neq0$, we have $X^{i}Y^{r-i}\in\ffpbar[\bar{U}]\cdot X^{r},$
and hence $\Sym^{\vec{r}}\ffpbar^{2}=\ffpbar[\bar{U}]\cdot X^{\vec{r}}.$

We now turn to the second statement to show that $\bar{\mathbb{F}}_{p}\cdot X^{\vec{r}}$
is the unique $U$-invariant line in ${\rm Sym}^{\vec{r}}\bar{\mathbb{F}}_{p}^{2}$.
As before, we think of ${\rm Sym}^{\vec{r}}\bar{\mathbb{F}}_{p}^{2}$
inside ${\rm Sym}^{r}\bar{\mathbb{F}}_{p}^{2}$, and consider a vector
$v\in{\rm Sym}^{\vec{r}}\bar{\mathbb{F}}_{p}^{2}\setminus\{\ffpbar\cdot X^{r}\}.$
We can write $v=\sum_{\vec{i}\preceq\vec{r}}\mu_{\vec{i}}X^{\sum_{j=0}^{n-1}i_{j}p^{j}}Y^{r-(\sum_{j=0}^{n-1}i_{j}p^{j})},$
where for the tuples $\vec{i}=(i_{0},\dots,i_{n-1})$ and $\vec{r}=(r_{0},\dots,r_{n-1})$
the notation $\vec{i}\preceq\vec{r}$ means $i_{j}\leq r_{j}$ for
every $j\in\{0,\dots,n-1\},$ and $\mu_{\vec{i}}$ are scalars such
that $\mu_{\vec{i}}\neq0$ for some $\vec{i}.$ Then, for any $\lambda\in\ffq$,
we have for $u(\lambda)\in U$ the following 
\begin{align*}
u(\lambda)\cdot v= & \sum_{\vec{i\preceq\vec{r}}}\mu_{\vec{i}}X^{\sum_{j=0}^{n-1}i_{j}p^{j}}(\lambda X+Y)^{r-(\sum_{j=0}^{n-1}i_{j}p^{j})}\\
= & \sum_{\vec{i}\preceq\vec{r}}\sum_{\vec{i}+\vec{j}\preceq\vec{r}}\mu_{\vec{i}}\lambda^{j}\binom{r-i}{j}X^{i+j}Y^{r-(i+j)}\\
= & \sum_{\vec{i}\preceq\vec{r}}\lambda^{i}\underset{:=v_{\vec{i}}}{\underbrace{\sum_{\vec{i}+\vec{j}\preceq\vec{r}}\mu_{\vec{j}}\binom{r-j}{i}X^{i+j}Y^{r-(i+j)}}},
\end{align*}
where in the second equality the integers $i$ and $j$ are the $p$-ary
sums corresponding to the tuples $\vec{i}$ and $\vec{j}$ respectively,
and in the third equality we have interchanged the tuples $\vec{i}$
and $\vec{j}.$ Now, since $v\notin\ffpbar\cdot X^{r},$ there is
some tuple $\vec{j_{0}}\precneqq\vec{r}$ such that $\mu_{\vec{j_{0}}}\neq0.$
We take such a $\vec{j_{0}}$ with the corresponding $p$-ary sum
$j_{0}$ being least. Consider two distinct tuples $\vec{i_{1}},\,\vec{i_{2}}\preceq\vec{r}-\vec{j_{0}}$
so that the polynomials $v_{\vec{i_{1}}}$and $v_{\vec{i_{2}}}$ are
non-zero and hence linearly independent (by comparing the monomials
of least degree in $X$). As a result, the subspace $\ffpbar[U]\cdot v$
has dimension at least two. 
\end{proof}
\begin{notation}
-- We write $\sigma_{\vec{r}}$ for the representation ${\rm Sym}^{\vec{r}}\bar{\mathbb{F}}_{p}^{2}$,
and $v_{\sigma_{\vec{r}}}$ for the generating vector $X^{\vec{r}}$
throughout this article. When $0\leq r<p$ we will write $r$ instead
of $\vec{r}$.
\end{notation}

\begin{rem}
\label{rem:Y^r generates weight as U-module} -- Similarly one can
show that $Y^{\vec{r}}$ also generates $\sigma_{\vec{r}}$ as an
$\ffpbar[U]$-module, and the line $\ffpbar\cdot Y^{\vec{r}}$ is
the unique $\bar{U}$-invariant (equivalently $I_{S}(1)^{\mathsf{T}}$-invariant)
line in $\sigma_{\vec{r}}$. Hence, the set $\{u\cdot(w_{0}\cdot v_{\sigma_{\vec{r}}})\,|\,u\in U_{S}(\calO)\}$
spans $\sigma_{\vec{r}}$.
\end{rem}

\subsection{Spherical Hecke algebra}

-- We recall that for a weight $\sigma_{\vec{r}}$ of $K_{0}:=\SL_{2}(\calO)$,
the \emph{spherical Hecke algebra} $\mathcal{H}(G_{S},K_{0},\sig):=\End_{G_{S}}(\ind_{K_{0}}^{G_{S}}(\sigma_{\vec{r}}))$
is generated by a single operator $\tau$ as a polynomial algebra
in one variable i.e. $\End_{G_{S}}(\ind_{K_{0}}^{G_{S}}(\sigma_{\vec{r}}))=\ffpbar[\tau]$.
In fact, we can compute the explicit action of $\tau$ on a standard
function $[g,v]\in\ind_{K_{0}}^{G_{S}}(\sigma_{\vec{r}})$, as follows
(see \cite[Corollaire 3.12]{Abdellatif-1}) :
\begin{multline}
\text{\fbox{\ensuremath{\tau([g,v])=\sum_{\lambda\in k_{F}^{2}}\Bigg[g\begin{pmatrix}\begin{array}{cc}
1 & A(\lambda)\\
0 & 1
\end{array}\end{pmatrix}\alpha_{0}^{-1},w_{0}U_{\vec{r}}\sigma_{\vec{r}}\bigg(\begin{pmatrix}\begin{array}{cc}
0 & 1\\
-1 & A(\lambda)
\end{array}\end{pmatrix}\bigg)v\Bigg]+\sum_{\mu\in k_{F}}\Bigg[g\begin{pmatrix}\begin{array}{cc}
1 & 0\\
\varpi_{F}A(\mu) & 1
\end{array}\end{pmatrix}\alpha_{0},U_{\vec{r}}v\Bigg].}}}\label{eq:action of tau on standard function}
\end{multline}
Here, for $\lambda=(\lambda_{0},\lambda_{1},\dots,\lambda_{m-1})\in k_{F}^{m}$
the notation $A(\lambda):=\sum_{j=0}^{m-1}[\lambda_{j}]\varpi_{F}^{j}$,
and $\alpha_{0}:=\mathrm{diag}(\varpi_{F}^{-1},\varpi_{F}).$ The
operator $U_{\vec{r}}:=U_{r_{0}}\otimes\cdots\otimes U_{r_{n-1}}\in\End_{\ffpbar}(\sigma_{\vec{r}})$
is such that $U_{r_{j}}\in\End_{\ffpbar}(\sigma_{r_{j}})$ and is
defined as follows : 
\begin{equation}
U_{r_{j}}(X^{l}Y^{r_{j}-l})=\begin{cases}
Y^{r_{j}} & \text{if }\,l=0\\
0 & \text{if }\,l\neq0
\end{cases}.\label{eq:formula for U_r}
\end{equation}

Now the compact induction ${\rm ind}_{K_{0}}^{G_{S}}(\sigma_{\vec{r}})$
is a left $\mathcal{H}(G_{S},K_{0},\sig)$-module. The mod $p$ irreducibles
of $K_{1}$ are denoted by $\sig^{\alpha}$. These have $\sig$ as
the underlying representation space on which an element $k_{1}=\alpha k_{0}\alpha^{-1}\in K_{1}$
(where $k_{0}\in K_{0}$) acts by $k_{0}.$ Since canonically we have
$\mathcal{H}(G_{S},K_{0},\sig)\simeq\mathcal{H}(G_{S},K_{1},\sig^{\alpha})$
and ${\rm ind}_{K_{1}}^{G_{S}}(\sig^{\alpha})\simeq{\rm ind}_{K_{0}}^{G_{S}}(\sig)^{\alpha}$
(see \cite[Propositions 3.6 and 3.23]{Abdellatif-1}) it suffices
to consider only the representation ${\rm ind}_{K_{0}}^{G_{S}}(\sigma_{\vec{r}})$
as a left $\mathcal{H}(G_{S},K_{0},\sig)$-module. The $\mathcal{H}(G_{S},K_{1},\sig^{\alpha})$-module
structure of ${\rm ind}_{K_{1}}^{G_{S}}(\sig^{\alpha})$ is essentially
same.

\subsection{Pro-$p$-Iwahori invariants of ${\rm ind}_{K_{0}}^{G_{S}}(\sigma_{\vec{r}})$}

-- We now recall some facts about the invariants $\ind_{K_{0}}^{G_{S}}(\sigma_{\vec{r}})^{I_{S}(1)}$.
It is known that it has a basis consisting of functions $\{f_{n}\,|\,n\in\mathbb{Z}\}$
such that the function $f_{n}$ is supported on $K_{0}\alpha_{0}^{-n}I_{S}(1)$
and satisfies
\[
f_{n}(\alpha_{0}^{-n})=\begin{cases}
w_{0}\cdot v_{\sigma_{\vec{r}}} & \text{if }\,n>0\\
v_{\sigma_{\vec{r}}} & \text{if }\,n\leq0
\end{cases}.
\]
Now, as $I_{S}(1)=U_{S}(\calO)\times T_{S}(1+\frakp)\times\bar{U}_{S}(\frakp)$,
and $\alpha_{0}$ normalizes $T_{S}$, we have $K_{0}\alpha_{0}^{-n}I_{S}(1)=K_{0}\alpha_{0}^{-n}\bar{U}_{S}(\frakp)$
and $K_{0}\alpha_{0}^{n}I_{S}(1)=K_{0}\alpha_{0}^{n}U_{S}(\calO)$
for $n>0$. Therefore, we have the decompositions 
\[
K_{0}\alpha_{0}^{-n}I_{S}(1)=\bigcup_{\bar{u}\in\bar{U}_{S}(\frakp)/\bar{U}_{S}(\frakp^{2n})}K_{0}\alpha_{0}^{-n}\bar{u}
\]
 and 
\[
K_{0}\alpha_{0}^{n}I_{S}(1)=\bigcup_{u\in U_{S}(\calO)/U_{S}(\frakp^{2n})}K_{0}\alpha_{0}^{n}u
\]
 for $n>0$, and hence for $n\in\mathbb{Z}$ we can write these $f_{n}$
as follows : 
\begin{equation}
\fbox{\ensuremath{f_{n}=\begin{cases}
\sum_{\bar{u}\in\bar{U}_{S}(\frakp)/\bar{U}_{S}(\frakp^{2n})}[\bar{u}\alpha_{0}^{n},w_{0}\cdot v_{\sigma_{\vec{r}}}] & \text{if }\,n>0\\
\sum_{u\in U_{S}(\calO)/U_{S}(\frakp^{-2n})}[u\alpha_{0}^{n},v_{\sigma_{\vec{r}}}] & \text{if }\,n\leq0
\end{cases}.} }\label{eq:decomposition of f_n}
\end{equation}

\section{An abstract linear-algebraic criterion}\label{subsec:A-technical-result}

In this section we prove a purely linear-algebraic/combinatorial lemma
that forms the abstract framework for the results of the subsequent
section. This result gives certain conditions to test when a graded
vector space with a finite collection of commuting linear operators
is a free module over the algebra generated by the operators.

\begin{lem}
\label{lem: technical lemma} --- Let $d\geq1$ and $V$ be a vector
space over a field $k$ and let
\[
V=\bigoplus_{\mathbf{n}\in\mathbb{Z}_{\ge0}^{d}}C_{\mathbf{n}}
\]
be a direct sum decomposition by non-trivial subspaces indexed by
multi-indices $\mathbf{n}=(n_{1},\dots,n_{d})$; set $|\bfn|:=n_{1}+\cdots+n_{d}$.
Let $T_{1},\dots,T_{d}$ be linear operators on $V$ which commute
pairwise \emph{:} $T_{i}T_{j}=T_{j}T_{i}$ for all $i,j$. Assume
the following \emph{:}
\begin{enumerate}
\item[\textit{\emph{(H1)}}]  Each $C_{\bfn}$ is finite dimensional and ${\rm dim}\,C_{\mathbf{n}}>\sum_{\{\bfm\,|\,|\bfm|<|\bfn|\}}{\rm dim}\,C_{\bfm}$.
\item[\textit{\emph{(H2)}}]  For each coordinate $j$,
\[
T_{j}(C_{\mathbf{0}})\subseteq C_{\mathbf{e}_{j}},\qquad\text{and}\qquad T_{j}|_{C_{\mathbf{0}}}:C_{\mathbf{0}}\to C_{\mathbf{e}_{j}}\text{ is injective}
\]
 where $\mathbf{e}_{j}$ denotes the $d$-tuple having $1$ in the
$j$-th coordinate and zero elsewhere.
\item[\textit{\emph{(H3)}}]  For every multi-index $\mathbf{n}\neq\mathbf{0}$ and each $j\in\{1,\dots,d\}$,
\[
T_{j}(C_{\mathbf{n}})\subseteq C_{\mathbf{n}-\mathbf{e}_{j}}\oplus C_{\mathbf{n}}\oplus C_{\mathbf{n}+\mathbf{e}_{j}},
\]
where $C_{\mathbf{n}-\mathbf{e}_{j}}$ is taken to be $0$ if $n_{j}=0$.
\item[\textit{\emph{(H4)}}]  For each $N\geq0$, write
\[
B_{N}:=\bigoplus_{|\mathbf{n}|\le N}C_{\mathbf{n}}.
\]
Then, for each $j$ and each $N\geq0$ the following holds \emph{:}
if $f\in B_{N+1}$ and $T_{j}(f)\in B_{N+1}$, then $f\in B_{N}$.
\item[\textit{\emph{(H5)}}]  For each $\bfn$ with $|\bfn|\geq2$, the collection of subspaces
\[
\{(\pi_{\bfn}\circ\bfT^{\bfn-\bfj})(C_{\bfj})\,|\,\bfj<\bfn\,\,\,\text{ coordinatewise }\}\qquad(\text{where }\pi_{\bfn}:V\to C_{\bfn}\text{ is the projection map})
\]
 are direct summands in $C_{\bfn}$. 
\end{enumerate}
Then there exists non-empty subsets $\ensuremath{A_{\mathbf{\bfn}}\subset C_{\mathbf{\bfn}}}$
for every $\bfn$ such that for any $N\geq0$ the collection 
\[
\{\mathbf{T}^{\mathbf{i}}(A_{\mathbf{j}})\,|\,|\mathbf{i}|+|\mathbf{j}|\le N\}\qquad(\text{where }\mathbf{T}^{\mathbf{i}}:=T_{1}^{i_{1}}\cdots T_{d}^{i_{d}})
\]
 is a mutually disjoint collection of sets and their union forms a
basis of $B_{N}$. Consequently, if $V$ is treated as a module over
the polynomial algebra $k[X_{1},\dots,X_{d}]$ where $X_{j}$ acts
by $T_{j}$, then the union $\bigsqcup_{\bfn}A_{\bfn}$ is a module
basis of $V$.
\end{lem}

\begin{proof}
We will construct the sets $A_{\mathbf{n}}$ inductively. Set $A_{\mathbf{0}}\subset C_{\mathbf{0}}$
to be a basis. Since ${\rm dim}\,C_{\mathbf{e}_{j}}>{\rm dim}\,C_{\mathbf{0}}$
for each $j$ we can take $A_{\mathbf{e}_{j}}\subset C_{\mathbf{e}_{j}}$
to be a non-empty set extending the linearly independent set $T_{j}(A_{\mathbf{0}})$
to a basis of $C_{\mathbf{e}_{j}}$. Thus we have constructed $A_{\mathbf{n}}$
for all $\bfn$ with $|\mathbf{n}|=1$. 

Now, we define ${\rm top}(f):={\rm max}\{|\mathbf{j}|\,|\,\pi_{\mathbf{j}}(f)\neq0\}$.
Then for each $j$ and each $N\geq0$ we have the following condition
\[
({\rm H4}^{\prime})\text{ --- }{\rm top}(f)=N\implies{\rm top}(T_{j}(f))=N+1.
\]
For $N=0$ this is essentially the hypothesis (H2). For $N\geq1$
if ${\rm top}(f)=N$ and ${\rm top}(T_{j}(f))\leq N$ (note that by
the hypothesis (H3) we know that ${\rm top}(T_{j}(f))\leq N+1$) then
by the hypothesis (H4) we know $f\in B_{N-1}$ and hence ${\rm top}(f)\leq N-1$,
a contradiction. Thus each $T_{j}$ is injective. Thus, if $f\in C_{\mathbf{j}}$,
then the \emph{top multi-index component} of $\mathbf{T}^{\mathbf{i}}(f)$
is \emph{precisely} $C_{\mathbf{i}+\mathbf{j}}$. Consequently, the
maps $\pi_{\mathbf{i}+\mathbf{j}}\circ\mathbf{T}^{\mathbf{i}}:C_{\mathbf{j}}\to C_{\mathbf{i}+\mathbf{j}}$
are injective. 

Now suppose $N\geq1$ and we have constructed non-empty subsets $A_{\mathbf{j}}\subset C_{\mathbf{j}}$
for all $\mathbf{j}$ with $|\mathbf{j}|\leq N$ such that 
\[
\bigsqcup_{|\mathbf{i}|+|\mathbf{j}|\leq\ell}\mathbf{T}^{\mathbf{i}}(A_{\mathbf{j}})
\]
is a basis for $B_{\ell}$ for all $\ell\leq N$. We consider the
following set 
\[
\mathcal{E}:=\bigsqcup_{\text{\ensuremath{\substack{|\mathbf{i}|+|\mathbf{j}| = N+1 \\
 |\mathbf{j}|\le N
}
}}} \mathbf{T}^{\mathbf{i}}(A_{\mathbf{j}}). 
\]
At first, we observe that the sets $\mathbf{T}^{\mathbf{i}}(A_{\mathbf{j}})$
appearing in the above union are indeed disjoint. Suppose $(\mathbf{i},\mathbf{j})\neq(\mathbf{i}^{\prime},\mathbf{j}^{\prime})$
with $|\mathbf{i}|+|\mathbf{j}|=|\mathbf{i}^{\prime}|+|\mathbf{j}^{\prime}|=N+1$.
We first consider the case when $\mathbf{j}=\mathbf{j}^{\prime}$
so that $|\mathbf{i}|=|\mathbf{i}^{\prime}|$ and $\mathbf{i}\neq\mathbf{i}^{\prime}$.
Then for $a,b\in A_{\mathbf{j}}$ we look at the top multi-index\emph{
}of $\mathbf{T}^{\mathbf{i}}(a)$ which is $\mathbf{i}+\mathbf{j}$;
whereas the top multi-index of the element $\mathbf{T}^{\mathbf{i}^{\prime}}(b)$
is $\mathbf{i}^{\prime}+\mathbf{j}$ and these are distinct. Next,
consider the case when $\mathbf{j}\neq\mathbf{j}^{\prime}$. In this
case if $\mathbf{i}+\mathbf{j}\neq\mathbf{i}^{\prime}+\mathbf{j}^{\prime}$
then for $a\in A_{\mathbf{j}}$ and $b\in A_{\mathbf{j}^{\prime}}$
the top components of $\mathbf{T}^{\mathbf{i}}(a)$ and $\mathbf{T}^{\mathbf{i}^{\prime}}(b)$
are in $C_{\mathbf{i}+\mathbf{j}}$ and $C_{\mathbf{i}^{\prime}+\mathbf{j}^{\prime}}$
respectively, which are different. Finally, we consider the case when
$\mathbf{j}\neq\mathbf{j}^{\prime}$ and $\bfn:=\mathbf{i}+\mathbf{j}=\mathbf{i}^{\prime}+\mathbf{j}^{\prime}$.
Then $\bfj,\bfj^{\prime}<\bfn$ coordinatewise. Hence, by the hypothesis
(H5) the top components of $\mathbf{T}^{\mathbf{i}}(a)$ and $\mathbf{T}^{\mathbf{i}^{\prime}}(b)$
are in different direct summands and hence they can never be equal.

Next, we show that $\mathcal{E}$ is linearly independent. Assume
a linear relation of the form 
\[
\sum_{\{(\bfi,\bfj)\,|\,|\bfi|+|\bfj|=N+1,\,|\bfj|\leq N\}}\sum_{a\in A_{\bfj}}\lambda_{\bfi,\bfj,a}\bfT^{\bfi}(a)=0.
\]
Fix any multi-index $\bfn$ with $|\bfn|=N+1$. We project onto $C_{\bfn}$
and obtain
\[
\sum_{\{\bfj<\bfn\,|\,|\bfj|\leq N\}}(\pi_{\bfn}\circ\bfT^{\bfn-\bfj})\left(\sum_{a\in A_{\bfj}}\lambda_{\bfn-\bfj,\bfj,a}a\right)=0.
\]
Now since for each $\bfj\in\{\bfj<\bfn\,|\,|\bfj|\leq N\}$ we know
$(\pi_{\bfn}\circ\bfT^{\bfn-\bfj})\left(\sum_{a\in A_{\bfj}}\lambda_{\bfn-\bfj,\bfj,a}a\right)\in(\pi_{\bfn}\circ\bfT^{\bfn-\bfj})(C_{\bfj})$
and these are direct summands, therefore by injectivity of $\pi_{\bfn}\circ\bfT^{\bfn-\bfj}:C_{\bfj}\to C_{\bfn}$
we have $\sum_{a\in A_{\bfj}}\lambda_{\bfn-\bfj,\bfj,a}a=0$ for all
such $\bfj$. Since by induction hypothesis $A_{\bfj}$ are linearly
independent we have each $\lambda_{\bfn-\bfj,\bfj,a}=0$. Thus we
have established linear independence of $\mathcal{E}$. 

Now consider the family 
\[
\bigsqcup_{\{(\bfi,\bfj)\,|\,|\bfi|+|\bfj|\leq N\}}\bfT^{\bfi}(A_{\bfj})\sqcup\bigsqcup_{\{(\bfi,\bfj)\,|\,|\bfj|\leq N,\,|\bfi|+|\bfj|=N+1\}}\bfT^{\bfi}(A_{\bfj}).
\]
This set is linearly independent since the second set is linearly
independent and has top components of degree $N+1$. Now we fix $\bfn$
with $|\bfn|=N+1$. We take $A_{\bfn}\subset C_{\bfn}$ to be a basis
of a complement of 
\[
\text{span}\left(\bigsqcup_{\{\bfj<\bfn\,|\,|\bfj|\leq N\}}\left(\pi_{\bfn}\circ\bfT^{\bfn-\bfj}\right)\left(A_{\bfj}\right)\right).
\]
Then the \emph{growth} condition (H1) together with injectivity of
the maps $\pi_{\bfn}\circ\bfT^{\bfn-\bfj}$ guarantee that $A_{\bfn}\neq\emptyset$.
Thus we obtain
\[
\bigsqcup_{\{(\bfi,\bfj)\,|\,|\bfi|+|\bfj|\leq N\}}\bfT^{\bfi}(A_{\bfj})\sqcup\bigsqcup_{\{(\bfi,\bfj)\,|\,|\bfj|\leq N,\,|\bfi|+|\bfj|=N+1\}}\bfT^{\bfi}(A_{\bfj})\sqcup\bigsqcup_{\{\bfn\,|\,|\bfn|=N+1\}}A_{\bfn}
\]
as a basis of $B_{N+1}$.
\end{proof}
\begin{rem}
\label{rem: d=00003D1 remark}-- For $d=1$ the hypothesis (H5) is
not required. Note that it was invoked in the proof in two instances
:
\begin{enumerate}
\item In showing that 
\[
\{T^{i}(A_{j})\,|\,i+j=N+1,\,j\leq N\}
\]
for $N\geq1$ forms a disjoint collection of sets. This can be proved
without (H5) as follows. Indeed, if $(i,j)\neq(i^{\prime},j^{\prime})$
then $j\neq j^{\prime}$. For $a\in A_{j}$ and $b\in A_{j^{\prime}}$
the equality 
\[
T^{i}(a)=T^{i^{\prime}}(b)
\]
gives a contradiction by combining injectivity of $T$ and the fact
that $T^{i-1}(a),T^{i^{\prime}-1}(b)\in\bigsqcup_{i+j\leq N}T^{i}(A_{j})$;
the disjointness of the family $\{T^{i}(A_{j})\,|\,i+j\leq N\}$ is
assumed in the induction hypothesis. 
\item In showing that
\[
\bigsqcup_{\{(i,j)\,|\,i+j\leq N\}}T^{i}(A_{j})\sqcup\bigsqcup_{\{(i,j)\,|\,i+j=N+1,\,j\leq N\}}T^{i}(A_{j})
\]
is linearly independent for $N\geq1$. Again this can be proved without
(H5) by a combined use of (H4) and the induction hypothesis : that
$\bigsqcup_{i+j\leq\ell}T^{i}(A_{j})$ is a basis of $B_{\ell}$ for
every $\ell\leq N$. Indeed, by way of contradiction suppose we have
a non-trivial linear combination of elements of the above set written
as $f_{1}+f_{2}=0$ where $f_{1}$ is a linear combination of elements
from the first set and $f_{2}$ is a linear combination of elements
from the second set. We can write $f_{2}$ as $T(f_{2}^{\prime})$
where $f_{2}^{\prime}$ is a linear combination of elements from $\bigsqcup_{j+i=N}T^{i}(A_{j})$
with same coefficients that appear in the combination representing
$f_{2}$. Thus $f_{2}^{\prime}\in B_{N}$ and $T(f_{2}^{\prime})=-f_{1}\in B_{N}$
and by condition (3) we have $f_{2}^{\prime}\in B_{N-1}$ since we
have assumed $N\geq1$. But $\bigsqcup_{j+i\leq N}T^{i}(A_{j})=\bigsqcup_{j+i\leq N-1}T^{i}(A_{j})\sqcup\bigsqcup_{j+i=N}T^{i}(A_{j})$
is a basis of $B_{N}$ and $\bigsqcup_{j+i\leq N-1}T^{i}(A_{j})$
is a basis of $B_{N-1}$ by the induction hypothesis. Therefore, all
coefficients in the linear combination representing $f_{2}^{\prime}$
(and hence $f_{2}$) are zero. Consequently, all coefficients representing
$f_{1}$ are also zero.
\end{enumerate}
Hence, we can choose $A_{N+1}$ to be a basis of a complement of 
\[
\text{span}\left(\pi_{N+1}\left(\bigsqcup_{\{(i,j)\,|\,i+j=N+1,\,j\leq N\}}T^{i}(A_{j})\right)\right)
\]
noting that $A_{N+1}\neq\emptyset$, since 
\[
\pi_{N+1}\left(\bigsqcup_{\{(i,j)\,|\,i+j=N+1,\,j\leq N\}}T^{i}(A_{j})\right)=\pi_{N+1}\left(T\left(\bigsqcup_{\{(i,j)\,|\,i+j=N,\,j\leq N\}}T^{i}(A_{j})\right)\right)\subset(\pi_{N+1}\circ T)(C_{N})\subsetneq C_{N+1}
\]
because $\pi_{N+1}\circ T$ is injective and ${\rm dim}\,C_{N}<{\rm dim}\,C_{N+1}$. 

However, when $d>1$ then without (H5) we can not ascertain the first
point made above. On a thorough scrutiny of the proof it will become
clear that (H5) is used really to overcome the problem of having to
show that : for two pairs $(\bfi,\bfj)$ and $(\bfi^{\prime},\bfj^{\prime})$
(with $|\bfj|,|\bfj^{\prime}|\leq N$, $|\bfi|+|\bfj|=|\bfi^{\prime}|+|\bfj^{\prime}|=N+1$)
satisfying $\bfi+\bfj=\bfi^{\prime}+\bfj^{\prime}$ and $\bfj\neq\bfj^{\prime}$,
the sets $\bfT^{\bfi}(A_{\bfj})$ and $\bfT^{\bfi^{\prime}}(A_{\bfj^{\prime}})$
are disjoint. More specifically, when $\bfi$ is of the form $(*,0,*,0,\dots)$
and $\bfi^{\prime}$ is of the form $(0,*,0,*,\dots)$, that is there
does not exist any coordinate $l$ for which $i_{l},i_{l}^{\prime}$
are both positive. Otherwise, if such a coordinate $l$ existed, then
we could have used injectivity of $T_{l}$ and the induction hypothesis
to conclude the disjointness. For example, if $d=2$ and $a\in C_{(0,1)},b\in C_{(1,0)}$
then without (H5) there is no way to ensure that $T_{1}(a)$ does
not coincide with $T_{2}(b)$ in $C_{(1,1)}$ (note that $T_{1}(a)\in C_{(0,1)}\oplus C_{(1,1)}$
and $T_{2}(b)\in C_{(1,0)}\oplus C_{(1,1)}$, thus if they were to
be equal then they must be elements of $C_{(1,1)}$). 
\end{rem}

\section{Freeness of $\protect\ind_{K_{0}}^{G_{S}}(\protect\sig)$ as a Hecke
module}\label{sec: main section}

In this section we take $V={\rm ind}_{K_{0}}^{G_{S}}(\sig)$ and $T=\tau$
in the framework of Lemma \ref{lem: technical lemma}. Here obviously
$d=1$. At first we introduce some new objects and notations borrowed
from \cite{Hu12,Xu_freeness}. First recall that we have the well
known Cartan-Iwahori decomposition : $G_{S}=\bigsqcup_{n\in\mathbb{Z}}K_{0}\alpha_{0}^{n}I_{S}(1)$.
For $n\geq0,$ we denote by $R_{n}^{+}(\sigma_{\vec{r}})$ (resp.
$R_{n}^{-}(\sigma_{\vec{r}})$) the subspace of functions in $\ind_{K_{0}}^{G_{S}}(\sigma_{\vec{r}})$
which are supported on $K_{0}\alpha_{0}^{n}I_{S}(1)=K_{0}\alpha_{0}^{n}U_{S}(\calO)$
(resp. $K_{0}\alpha_{0}^{-(n+1)}I_{S}(1)=K_{0}\alpha_{0}^{-(n+1)}\bar{U}_{S}(\frakp)$).
We write $R_{n}^{+}(\sigma_{\vec{r}})=[U_{S}(\calO)\alpha_{0}^{-n},\sigma_{\vec{r}}]$
for $n\geq0$, and $R_{n-1}^{-}(\sigma_{\vec{r}})=[\bar{U}_{S}(\frakp)\alpha_{0}^{n},\sigma_{\vec{r}}]$
for $n\geq1$. By convention $R_{-1}^{-}(\sigma_{\vec{r}}):=R_{0}^{+}(\sigma_{\vec{r}}).$
We also set $C_{0,\sigma_{\vec{r}}}:=R_{0}^{+}(\sigma_{\vec{r}}),\,C_{n,\sigma_{\vec{r}}}:=R_{n}^{+}(\sigma_{\vec{r}})\oplus R_{n-1}^{-}(\sigma_{\vec{r}})\text{ for }n\geq1$
and $B_{n,\sigma_{\vec{r}}}:=\bigoplus_{k\leq n}C_{k,\sigma_{\vec{r}}}\text{ for }n\geq0$.

As mentioned in Remark \ref{rem: d=00003D1 remark}, we need to check
the following conditions are satisfied
\begin{enumerate}
\item[(C1)]  Each $C_{n,\sig}$ is finite dimensional and ${\rm dim}\,C_{n,\sig}>\sum_{\{k\,|\,k<n\}}C_{k,\sig}$.
\item[(C2)]  $\tau(C_{0,\sig})\subset C_{1,\sig}$ and $\tau|_{C_{0,\sig}}:C_{0,\sig}\to C_{1,\sig}$
is injective.
\item[(C3)]  For every $n\geq1$ we have $\tau(C_{n\sig})\subset C_{n-1,\sig}\oplus C_{n,\sig}\oplus C_{n+1,\sig}$.
\item[(C4)]  For each $n\geq0$ : if $f\in B_{n+1,\sig}$ and $\tau(f)\in B_{n+1,\sig}$
then $f\in B_{n,\sig}$.
\end{enumerate}
See the following remark for (C1).

\begin{rem}
\label{rem: condition C1} -- It is often useful to interpret the
subspaces $C_{n,\sig}$ and $B_{n,\sig}$ in terms of the Bruhat-Tits
tree of ${\rm SL}_{2}(F)$. For a basic introduction to Bruhat-Tits
theory with the example of $\SL_{2}$ worked out in full detail one
can see \cite[Section 4.2]{Rab}; also the recent work \cite{Herbert et al}
works out the example $\SL_{n}$ with copious details. Note that $C_{0,\sig}$
is the subspace of all functions supported on the central vertex corresponding
to the subgroup $K_{0}.$ For $n\geq1$ the subspace $C_{n,\sig}$
consists of all functions supported on vertices which are at a distance
of $2n$ from the central vertex. Since each vertex of the tree has
degree $q:=|k_{F}|$, we know that ${\rm dim}\,C_{n,\sig}=(q+1)q^{2n-1}\cdot{\rm dim}\,\sig$
for $n\geq1$ and ${\rm dim}\,C_{0,\sig}={\rm dim}\,\sig$. Hence
the inequality 
\[
{\rm dim}\,C_{n,\sig}>\sum_{m=0}^{n-1}{\rm dim}\,C_{m,\sig}
\]
is clear when $n=1$. For $n\geq2$ we have to show 
\[
(q+1)q^{2n-1}>1+(q+1)\cdot q\cdot\frac{q^{2(n-1)}-1}{q^{2}-1},
\]
which can be easily proved using elementary algebra noting that $q\geq2$.
Thus the condition (C1) above checks out.
\end{rem}

\begin{rem}
\label{rem : I_S(1) invariants of the sphere R_n(sigma)} -- Also
note that for $n\geq0$ both the spaces $R_{n}^{+}(\sig)$ and $R_{n-1}^{-}(\sig)$
are $I_{S}(1)$-stable. For $n\geq0$ we have $f_{-n}\in R_{n}^{+}(\sigma_{\vec{r}})^{I_{S}(1)}$,
and for $n\geq1$ we have $f_{n}\in R_{n-1}^{-}(\sigma_{\vec{r}})^{I_{S}(1)}$.
Also, every function in $R_{n}^{+}(\sigma_{\vec{r}})^{I_{S}(1)}$
(resp. $R_{n-1}^{-}(\sigma_{\vec{r}})^{I_{S}(1)}$) is determined
by its value on $\alpha_{0}^{n}$ (resp. $\alpha_{0}^{-n}$). Hence,
the spaces $R_{n}^{+}(\sigma_{\vec{r}})^{I_{S}(1)}$ and $R_{n-1}^{-}(\sigma_{\vec{r}})^{I_{S}(1)}$
are one dimensional and generated by $f_{-n}$ and $f_{n}$ respectively.
\end{rem}

We now state the main theorem of this article. 
\begin{thm}
\label{thm:freeness theorem} --- The compactly induced representation
$\ind_{K_{0}}^{G_{S}}(\sig)$ is a free module of infinite rank over
the spherical Hecke algebra $\mathcal{H}(G_{S},K_{0},\sig)$.
\end{thm}

To prove this we will show that the Hecke operator $\tau$ satisfies
the conditions (C2)-(C4) above. 

\subsection{Action of $\tau$ on the subspaces $C_{n,\protect\sig}$}

We will now show that the Hecke operator $\tau$ satisfies the conditions
(C2) and (C3).
\begin{lem}
\emph{\label{lem: action of tau on 2n circles} --- }We have the
following :
\begin{enumerate}
\item[(i)] \emph{ $\tau(R_{0}^{+}(\sigma_{\vec{r}}))\subseteq R_{1}^{+}(\sigma_{\vec{r}})\oplus R_{0}^{-}(\sigma_{\vec{r}}).$}
\item[(ii)] \emph{ $\tau(R_{n}^{+}(\sigma_{\vec{r}}))\subseteq R_{n-1}^{+}(\sigma_{\vec{r}})\oplus R_{n}^{+}(\sigma_{\vec{r}})\oplus R_{n+1}^{+}(\sigma_{\vec{r}})$,
}for\emph{ $n\geq1.$}
\item[(iii)]  $\tau(R_{n}^{-}(\sigma_{\vec{r}}))\subseteq R_{n-1}^{-}(\sigma_{\vec{r}})\oplus R_{n}^{-}(\sigma_{\vec{r}})\oplus R_{n+1}^{-}(\sigma_{\vec{r}})$,
for $n\geq0.$
\end{enumerate}
Thus $\tau(C_{0,\sig})\subseteq C_{1,\sig}$ and $\tau(C_{n,\sig})\subseteq C_{n-1,\sig}\oplus C_{n,\sig}\oplus C_{n+1,\sig}$
for $n\geq1$,
\end{lem}

\begin{proof}
(i) It suffices to compute the action of $\tau$ on a standard function
$[u,v]\in R_{0}^{+}(\sigma_{\vec{r}})$, where $u\in U_{S}(\calO)$.
Using the formula \ref{eq:action of tau on standard function} the
required containment follows from the simple observation that if $\bar{u}\in\bar{U}_{S}(\frakp)$,
then $u\bar{u}\in I_{S}(1)$, so that we can write $u\bar{u}=\bar{u}_{1}tu_{2}$,
where $\bar{u}_{1}\in\bar{U}_{S}(\frakp),$ $t\in T_{S}(1+\frakp),$
and $u_{2}\in U_{S}(\calO).$ But then again we can write $u_{2}\alpha_{0}=\alpha_{0}u_{2}^{\prime}$
with $u_{2}^{\prime}\in U_{S}(\frakp^{2})\subset K_{0}.$ Hence, the
sum $\sum_{\mu}[u\bar{u}(\varpi_{F}A(\mu))\alpha_{0},U_{\vec{r}}v]\in R_{0}^{-}(\sigma_{\vec{r}})$. 

(ii) Again, take a standard function $[u(x)\alpha_{0}^{-n},v]\in R_{n}^{+}(\sigma_{\vec{r}})$
where $u(x)$ denotes an upper unipotent matrix with the top right
entry $x\in\calO$. Here, we look at the formula \ref{eq:action of tau on standard function}
and consider the elements of the form $u(x)\alpha_{0}^{-n}u(A(\lambda))\alpha_{0}^{-1}$
and $u(x)\alpha_{0}^{-n}\bar{u}(\varpi_{F}A(\mu))\alpha_{0}$ to determine
which $U_{S}(\calO)$-$K_{0}$ double coset they belong to. At first,
note that $u(x)\alpha_{0}^{-n}u(A(\lambda))\alpha_{0}^{-1}=u(x+\varpi_{F}^{2n}A(\lambda))\alpha_{0}^{-(n+1)}\in U_{S}(\calO)\alpha_{0}^{-(n+1)}.$
Next, for $\mu\in k_{F}$ we can easily see by using elementary row
and column reduction that : 
\[
\alpha_{0}^{-n}\bar{u}(\varpi_{F}A(\mu))\alpha_{0}=\begin{pmatrix}\begin{array}{cc}
\varpi_{F}^{n-1} & 0\\
\varpi_{F}^{-n}A(\mu) & \varpi_{F}^{-n+1}
\end{array}\end{pmatrix}\in\begin{cases}
U_{S}(\calO)\alpha_{0}^{-(n-1)}K_{0} & \text{if }\,\mu=0\\
U_{S}(\calO)\alpha_{0}^{-n}K_{0} & \text{if }\,\mu\neq0
\end{cases}.
\]
As a result we have \emph{$\tau(R_{n}^{+}(\sigma_{\vec{r}}))\subseteq R_{n-1}^{+}(\sigma_{\vec{r}})\oplus R_{n}^{+}(\sigma_{\vec{r}})\oplus R_{n+1}^{+}(\sigma_{\vec{r}})$,
}for\emph{ $n\geq1$}.

(iii) As before we take a standard function $[\bar{u}(y)\alpha_{0}^{n+1},v]\in R_{n}^{-}(\sigma_{\vec{r}})$
where $\bar{u}(y)$ denotes a lower unipotent matrix with the bottom
left entry $y\in\frakp$. In this case also we look at the formula
\ref{eq:action of tau on standard function} and consider the elements
$\bar{u}(y)\alpha_{0}^{n+1}u(A(\lambda))\alpha_{0}^{-1}$ and $\bar{u}(y)\alpha_{0}^{n+1}\bar{u}(\varpi_{F}A(\mu))\alpha_{0}$.
Then we determine which $\bar{U}_{S}(\frakp)$-$K_{0}$ double coset
they belong to. First, note that $\bar{u}(y)\alpha_{0}^{n+1}\bar{u}(\varpi_{F}A(\mu))\alpha_{0}=\bar{u}(y+\varpi_{F}^{2n+3}A(\mu))\alpha_{0}^{n+2}\in\bar{U}_{S}(\frakp)\alpha_{0}^{n+2}$.
Next, for $\lambda\in k_{F}^{2}$ we have by row and column reduction
the following : 
\[
\alpha_{0}^{n+1}u(A(\lambda))\alpha_{0}^{-1}=\begin{pmatrix}\begin{array}{cc}
\varpi_{F}^{-n} & \varpi_{F}^{-n-2}A(\lambda)\\
0 & \varpi_{F}^{n}
\end{array}\end{pmatrix}\in\begin{cases}
\bar{U}_{S}(\frakp)\alpha_{0}^{n+2}K_{0} & \text{if }\,A(\lambda)\in\calO^{\times}\\
\bar{U}_{S}(\frakp)\alpha_{0}^{n+1}K_{0} & \text{if }\,A(\lambda)\in\frakp\setminus\frakp^{2}\\
\bar{U}_{S}(\frakp)\alpha_{0}^{n}K_{0} & \text{if }\,A(\lambda)\in\frakp^{2}
\end{cases}.
\]
Hence, we have $\tau(R_{n}^{-}(\sigma_{\vec{r}}))\subseteq R_{n-1}^{-}(\sigma_{\vec{r}})\oplus R_{n}^{-}(\sigma_{\vec{r}})\oplus R_{n+1}^{-}(\sigma_{\vec{r}})$,
for $n\geq0.$
\end{proof}
\begin{lem}
\label{lem:tau injective on C_0} --- The map $\tau:C_{0,\sig}\to C_{1,\sig}$
is injective.
\end{lem}

\begin{proof}
Note that $\bar{U}_{S}(\calO)$-translates of $[I,X^{\vec{r}}]$ generates
the space $C_{0,\sig}$, by Proposition \ref{prop:spanning set for weights}.
Since the right hand side of the formula \ref{eq:action of tau on standard function}
has standard functions supported on distinct cosets, it suffices to
show that for any non-zero $v\in\sig$ there exists some $\lambda\in k_{F}$
such that 
\[
U_{\vec{r}}\left(\sig\left(\left(\begin{array}{cc}
0 & 1\\
-1 & A(\lambda)
\end{array}\right)\right)v\right)\neq0.
\]
As in the proof of Proposition \ref{prop:spanning set for weights}
we consider $\Sym^{\vec{r}}\ffpbar^{2}$ as an $\SL_{2}(k_{F})$-subrepresentation
of $\Sym^{r}\ffpbar^{2}$ via the map $v_{0}\otimes v_{1}\otimes\cdots\otimes v_{n-1}\mapsto v_{0}v_{1}^{p}\cdots v_{n-1}^{p^{n-1}}$
and work with the monomial basis 
\[
\left\{ X^{r-i}Y^{i}\,:\,i=\sum_{j=0}^{n-1}i_{j}p^{j}\leq r=\sum_{j=0}^{n-1}r_{j}p^{j},\,i_{j},r_{j}\in\{0,\dots,p-1\},\,i_{j}\leq r_{j}\text{ for every }j\right\} .
\]
Now we write $v=\sum_{i}c_{i}X^{r-i}Y^{i}$ as a linear combination
with the admissible monomials and coefficients $c_{i}\in\ffpbar$
not all zero. For any $a\in k_{F}$ the coefficient of $X^{r}$ in
$u(a)\cdot v$ is $P_{v}(a)=\sum_{i}c_{i}a^{i}$ which is a polynomial
(in $a$) of degree $r<q$ with coefficients $c_{i}\in\ffpbar$ not
all zero. Thus there exists some $a\in k_{F}$ so that $P_{v}(a)\neq0$.
We let $A(\lambda):=[a]$ where $[\cdot]$ denotes the Teichmuller
lift as usual. Therefore, the coefficient of $X^{r}$ in $u(A(\lambda))\cdot v$
is $P_{v}(a)$ which is non-zero. Hence, the $Y^{r}$ coefficient
of 
\[
\left(\begin{array}{cc}
0 & 1\\
-1 & A(\lambda)
\end{array}\right)\cdot v=w_{0}u(A(\lambda))\cdot v
\]
 is non-zero. So $U_{\vec{r}}\left(\sig\left(\left(\begin{array}{cc}
0 & 1\\
-1 & A(\lambda)
\end{array}\right)\right)v\right)\neq0$ by eq. \ref{eq:formula for U_r} as required.
\end{proof}

\subsection{Action of $\tau$ strictly increases the ${\rm top}$ value}

Next, we will show that $\tau$ satisfies the condition (C4). But
we first need to explicitly compute the action of $\tau$ on $\ind_{K_{0}}^{G_{S}}(\sigma_{\vec{r}})^{I_{S}(1)}$
using the explicit formulae \ref{eq:action of tau on standard function}
and \ref{eq:decomposition of f_n}.
\begin{lem}
\label{lem:tau(f_n)} --- We have
\begin{enumerate}
\item[(i)]  $\tau(f_{0})=f_{-1}+\lambda_{\sigma_{\vec{r}}}f_{1}$, where $\lambda_{\sigma_{\vec{r}}}=\begin{cases}
0 & \text{if }\,\vec{r}\neq\vec{0}\\
1 & \text{if }\,\vec{r}=\vec{0}
\end{cases}.$
\item[(ii)]  For $n\neq0$ we have $\tau(f_{n})=c_{n}f_{n}+f_{n+\delta(n)}$,
where $c_{n}$ is a scalar and $\delta(n)=\begin{cases}
1 & \text{if }\,n>0\\
-1 & \text{if }\,n<0
\end{cases}.$
\end{enumerate}
\end{lem}

\begin{proof}
(i) At first, we consider the case when $\sigma_{\vec{r}}\neq1$.
Then $f_{0}=[I_{2},X^{\vec{r}}]$, and we have :
\begin{align*}
\tau(f_{0})= & \tau([I_{2},X^{\vec{r}}])\\
= & \sum_{\lambda\in k_{F}^{2}}\Bigg[\begin{pmatrix}\begin{array}{cc}
1 & A(\lambda)\\
0 & 1
\end{array}\end{pmatrix}\alpha_{0}^{-1},w_{0}U_{\vec{r}}\sigma_{\vec{r}}\bigg(\begin{pmatrix}\begin{array}{cc}
0 & 1\\
-1 & A(\lambda)
\end{array}\end{pmatrix}\bigg)(X^{\vec{r}})\Bigg]+\sum_{\mu\in k_{F}}\Bigg[\begin{pmatrix}\begin{array}{cc}
1 & 0\\
\varpi_{F}A(\mu) & 1
\end{array}\end{pmatrix}\alpha_{0},U_{\vec{r}}X^{\vec{r}}\Bigg]\\
= & \sum_{\lambda\in k_{F}^{2}}\Bigg[\begin{pmatrix}\begin{array}{cc}
1 & A(\lambda)\\
0 & 1
\end{array}\end{pmatrix}\alpha_{0}^{-1},w_{0}U_{\vec{r}}((-1)^{r}Y^{\vec{r}})\Bigg]\\
= & \sum_{\lambda\in k_{F}^{2}}\Bigg[\begin{pmatrix}\begin{array}{cc}
1 & A(\lambda)\\
0 & 1
\end{array}\end{pmatrix}\alpha_{0}^{-1},(-1)^{r}w_{0}\cdot Y^{\vec{r}}\Bigg]\\
= & \sum_{\lambda\in k_{F}^{2}}\Bigg[\begin{pmatrix}\begin{array}{cc}
1 & A(\lambda)\\
0 & 1
\end{array}\end{pmatrix}\alpha_{0}^{-1},(-1)^{2r}X^{\vec{r}}\Bigg]\\
= & f_{-1}.
\end{align*}
Next, let $\sigma_{\vec{r}}=1$. Then, we have :
\begin{align*}
\tau(f_{0})= & \tau([I_{2},1])\\
= & \sum_{\lambda\in k_{F}^{2}}\begin{pmatrix}\begin{array}{cc}
1 & A(\lambda)\\
0 & 1
\end{array}\end{pmatrix}\alpha_{0}^{-1}\cdot[I_{2},1]+\sum_{\mu\in k_{F}}\begin{pmatrix}\begin{array}{cc}
1 & 0\\
\varpi_{F}A(\mu) & 1
\end{array}\end{pmatrix}\alpha_{0}\cdot[I_{2},1]\\
= & \sum_{u\in U_{S}(\calO)/U_{S}(\frakp^{2})}[u\alpha_{0}^{-1},1]+\sum_{\bar{u}\in\bar{U}_{S}(\frakp)/\bar{U}_{S}(\frakp^{2})}[\bar{u}\alpha_{0},1]\\
= & f_{1}+f_{-1}.
\end{align*}

(ii) At first, we take $n=-m$ to be a negative integer. We will show
that 
\[
\tau(f_{n})=c_{n}f_{n}+f_{n-1}
\]
for some $c_{n}\in\ffpbar$. Now, $f_{-m}$ is $I_{S}(1)$-invariant
and as $\tau$ is a $G_{S}$-intertwiner, $\tau(f_{-m})$ is also
$I_{S}(1)$-invariant. Since each $R_{j}^{+}(\sigma_{\vec{r}})$ for
$j\geq0$ is $I_{S}(1)$-stable we have 
\[
\tau(f_{-m})\in(R_{m-1}^{+}(\sigma_{\vec{r}})\oplus R_{m}^{+}(\sigma_{\vec{r}})\oplus R_{m+1}^{+}(\sigma_{\vec{r}}))^{I_{S}(1)}=R_{m-1}^{+}(\sigma_{\vec{r}})^{I_{S}(1)}\oplus R_{m}^{+}(\sigma_{\vec{r}})^{I_{S}(1)}\oplus R_{m+1}^{+}(\sigma_{\vec{r}})^{I_{S}(1)},
\]
and so we can write 
\[
\tau(f_{-m})=c_{m-1}f_{-m+1}+c_{m}f_{-m}+c_{m+1}f_{-m-1}
\]
 for some $c_{m-1},c_{m},c_{m+1}\in\ffpbar$. Now, by the formula
\ref{eq:decomposition of f_n} and part (1), we have 
\[
\tau(f_{-m})=\tau\bigg(\sum_{u\in U_{S}(\calO)/U_{S}(\frakp^{2m})}u\alpha_{0}^{-m}\cdot f_{0}\bigg)=\sum_{u\in U_{S}(\calO)/U_{S}(\frakp^{2m})}u\alpha_{0}^{-m}(f_{-1}+\lambda_{\vec{r}}f_{1}).
\]
So, to find the constants $c_{m-1}$ and $c_{m+1}$ we evaluate $\tau(f_{-m})$
above at $\alpha_{0}^{m-1}$ and $\alpha_{0}^{m+1}$. For this we
at first need to find which $K_{0}$-$I_{S}(1)$ double coset the
elements $\alpha_{0}^{m-1}u\alpha_{0}^{-m}$ and $\alpha_{0}^{m+1}u\alpha_{0}^{-m}$
lie in, for $u\in U_{S}(\calO)/U_{S}(\frakp^{2m})$. Let $u$ be an
upper unipotent with the top right entry $x\in\calO/\frakp^{2m}$,
and we have 
\[
\alpha_{0}^{m-1}\begin{pmatrix}\begin{array}{cc}
1 & x\\
0 & 1
\end{array}\end{pmatrix}\alpha_{0}^{-m}=\begin{pmatrix}\begin{array}{cc}
\varpi_{F} & x\varpi_{F}^{1-2m}\\
0 & \varpi_{F}^{-1}
\end{array}\end{pmatrix}\in\begin{cases}
K_{0}\alpha_{0}^{-1}I_{S}(1) & \text{if }\,x\in\frakp^{2m-2}\\
K_{0}\alpha_{0}^{l(x)}I_{S}(1) & \text{if }\,x\in\calO\setminus\frakp^{2m-2}
\end{cases},
\]
where $l(x)<-1;$ the above containments can be seen by elementary
row and column reduction. For instance if $x=a\varpi_{F}^{2m-2}$
with $a\in\calO$ then : 
\[
\begin{pmatrix}\begin{array}{cc}
1 & -a\\
0 & 1
\end{array}\end{pmatrix}\begin{pmatrix}\begin{array}{cc}
\varpi_{F} & a\varpi_{F}^{-1}\\
0 & \varpi_{F}^{-1}
\end{array}\end{pmatrix}=\alpha_{0}^{-1},
\]
and if $x=a\varpi_{F}^{2m-d}\in\calO$ with $d\geq3$ and $a\in\calO^{\times}$
then : 
\begin{align*}
\begin{pmatrix}\begin{array}{cc}
a & 0\\
0 & a^{-1}
\end{array}\end{pmatrix}\begin{pmatrix}\begin{array}{cc}
0 & -1\\
1 & 0
\end{array}\end{pmatrix}\begin{pmatrix}\begin{array}{cc}
1 & 0\\
-a^{-1}\varpi_{F}^{d-2} & 1
\end{array}\end{pmatrix}\begin{pmatrix}\begin{array}{cc}
\varpi_{F} & a\varpi_{F}^{1-d}\\
0 & \varpi_{F}^{-1}
\end{array}\end{pmatrix}\begin{pmatrix}\begin{array}{cc}
1 & 0\\
-a^{-1}\varpi_{F}^{d} & 1
\end{array}\end{pmatrix} & =\alpha_{0}^{1-d}.
\end{align*}
Then, we have 
\begin{align*}
\tau(f_{-m})(\alpha_{0}^{m-1})= & \sum_{u\in U_{S}(\calO)/U_{S}(\frakp^{2m})}f_{-1}(\alpha_{0}^{m-1}u\alpha_{0}^{-m})+\lambda_{\vec{r}}\sum_{u\in U_{S}(\calO)/U_{S}(\frakp^{2m})}f_{1}(\alpha_{0}^{m-1}u\alpha_{0}^{-m})\\
= & \lambda_{\vec{r}}\sum_{u\in U_{S}(\frakp^{2m-2})/U_{S}(\frakp^{2m})}f_{1}\left(\underset{{\rm write\,as\,}u_{1}\alpha_{0}^{-1}\,{\rm with\,}u_{1}\in U_{S}(\calO)/U_{S}(\frakp^{2})}{\underbrace{\alpha_{0}^{m-1}u\alpha_{0}^{-m}}}\right)\\
= & \lambda_{\vec{r}}\sum_{u_{1}\in U_{S}(\calO)/U_{S}(\frakp^{2})}f_{1}(u_{1}\alpha_{0}^{-1})\\
= & \lambda_{\vec{r}}\sum_{\mu\in k_{F}^{2}}\begin{pmatrix}\begin{array}{cc}
1 & [\mu_{0}]+[\mu_{1}]\varpi_{F}\\
0 & 1
\end{array}\end{pmatrix}\cdot f_{1}(\alpha_{0}^{-1})\\
= & \lambda_{\vec{r}}\sum_{\mu\in k_{F}^{2}}\begin{pmatrix}\begin{array}{cc}
1 & [\mu_{0}]+[\mu_{1}]\varpi_{F}\\
0 & 1
\end{array}\end{pmatrix}\cdot(w_{0}\cdot v_{\sigma_{\vec{r}}})\\
= & \lambda_{\vec{r}}\sum_{(\mu_{0},\mu_{1})\in k_{F}\times k_{F}}\begin{pmatrix}\begin{array}{cc}
1 & [\mu_{0}]\\
0 & 1
\end{array}\end{pmatrix}\cdot(w_{0}\cdot v_{\sigma_{\vec{r}}})\\
= & 0.
\end{align*}
Hence, $c_{m-1}=0$. Next, we take $u$ to be an upper unipotent with
the top right entry $x\in\calO/\frakp^{2m}$, and we have 
\[
\alpha_{0}^{m+1}\begin{pmatrix}\begin{array}{cc}
1 & x\\
0 & 1
\end{array}\end{pmatrix}\alpha_{0}^{-m}=\begin{pmatrix}\begin{array}{cc}
\varpi_{F}^{-1} & x\varpi_{F}^{-2m-1}\\
0 & \varpi_{F}
\end{array}\end{pmatrix}\in\begin{cases}
K_{0}\alpha_{0}I_{S}(1) & \text{if }\,x\in\frakp^{2m}\\
K_{0}\alpha_{0}^{l^{\prime}(x)}I_{S}(1) & \text{if }\,x\in\calO\setminus\frakp^{2m}
\end{cases},
\]
where $l^{\prime}(x)<-1.$ The above containments can be proved as
before. Therefore, we have : 
\begin{align*}
\tau(f_{-m})(\alpha_{0}^{m+1})= & \sum_{u\in U_{S}(\calO)/U_{S}(\frakp^{2m})}f_{-1}(\alpha_{0}^{m+1}u\alpha_{0}^{-m})+\lambda_{\vec{r}}\sum_{u\in U_{S}(\calO)/U_{S}(\frakp^{2m})}f_{1}(\alpha_{0}^{m+1}u\alpha_{0}^{-m})\\
= & f_{-1}(\alpha_{0})=v_{\sigma_{\vec{r}}}.
\end{align*}
Hence, $c_{m+1}=1.$ 

Next, we consider the case when $n=m+1$ is positive i.e. $m\geq0$.
Then, $f_{n}=f_{m+1}\in R_{m}^{-}(\sigma_{\vec{r}})$, hence
\[
\tau(f_{m+1})=c_{m-1}f_{m}+c_{m+1}f_{m+1}+c_{m+2}f_{m+2},
\]
and we want to show that $c_{m-1}=0$ and $c_{m+2}=1.$ Now, by the
formula \ref{eq:decomposition of f_n} and part (1), we get 
\begin{align*}
\tau(f_{m+1})= & \tau\Big(\sum_{\bar{u}\in\bar{U}_{S}(\frakp)/\bar{U}_{S}(\frakp^{2m+2})}[\bar{u}\alpha_{0}^{m+1},w_{0}\cdot v_{\sigma_{\vec{r}}}]\Big)=\sum_{\bar{u}\in\bar{U}_{S}(\frakp)/\bar{U}_{S}(\frakp^{2m+2})}\bar{u}\alpha_{0}^{m+1}w_{0}\tau(f_{0})\\
= & \sum_{\bar{u}\in\bar{U}_{S}(\frakp)/\bar{U}_{S}(\frakp^{2m+2})}\bar{u}\alpha_{0}^{m+1}w_{0}(f_{-1}+\lambda_{\sigma_{\vec{r}}}f_{1})
\end{align*}
So, to find the constants $c_{m-1}$ and $c_{m+2}$ we evaluate $\tau(f_{m+1})$
at $\alpha_{0}^{-m}$ and $\alpha_{0}^{-m-2}$ respectively. As before,
we at first find which $K_{0}$-$I_{S}(1)$ double coset the elements
$\alpha_{0}^{m}u\alpha_{0}^{-m-1}$ and $\alpha_{0}^{m+2}u\alpha_{0}^{-m-1}$
lie in, for $u\in U_{S}(\frakp)/U_{S}(\frakp^{2m+2}).$ Let $u$ be
an upper unipotent with the top right entry $x\in\frakp/\frakp^{2m+2}$.
Then, we have 
\[
\alpha_{0}^{m}\begin{pmatrix}\begin{array}{cc}
1 & x\\
0 & 1
\end{array}\end{pmatrix}\alpha_{0}^{-m-1}=\begin{pmatrix}\begin{array}{cc}
\varpi_{F} & x\varpi_{F}^{-2m-1}\\
0 & \varpi_{F}^{-1}
\end{array}\end{pmatrix}\in\begin{cases}
K_{0}\alpha_{0}^{-1}I_{S}(1) & \text{if }\,x\in\frakp^{2m}\\
K_{0}\alpha_{0}^{l(x)}I_{S}(1) & \text{if }\,x\in\frakp\setminus\frakp^{2m}
\end{cases},
\]
where $l(x)<-1$. As a result we have 
\begin{align*}
\tau(f_{m+1})(\alpha_{0}^{-m})= & \sum_{\bar{u}\in\bar{U}_{S}(\frakp)/\bar{U}_{S}(\frakp^{2m+2})}f_{-1}(\alpha_{0}^{-m}\bar{u}\alpha_{0}^{m+1}w_{0})+\lambda_{\vec{r}}\sum_{\bar{u}\in\bar{U}_{S}(\frakp)/\bar{U}_{S}(\frakp^{2m+2})}f_{1}(\alpha_{0}^{-m}\bar{u}\alpha_{0}^{m+1}w_{0})\\
= & \lambda_{\vec{r}}\sum_{u\in U_{S}(\frakp^{2m})/U_{S}(\frakp^{2m+2})}w_{0}\cdot f_{1}\left(\underset{{\rm write\,}{\rm as\,}u_{1}\alpha_{0}^{-1}\,{\rm with\,}u_{1}\in U_{S}(\calO)/U_{S}(\frakp^{2})}{\underbrace{\alpha_{0}^{m}u\alpha_{0}^{-m-1}}}\right)\\
= & \lambda_{\vec{r}}\sum_{u_{1}\in U_{S}(\calO)/U_{S}(\frakp^{2})}(w_{0}u_{1}w_{0})\cdot v_{\sigma_{\vec{r}}}\\
= & \lambda_{\vec{r}}\sum_{(\mu_{0},\mu_{1})\in k_{F}^{2}}(w_{0}\begin{pmatrix}\begin{array}{cc}
1 & [\mu_{0}]\\
0 & 1
\end{array}\end{pmatrix}w_{0})\cdot v_{\sigma_{\vec{r}}}=0,
\end{align*}
and hence $c_{m-1}=0.$ Next, for $u$ an upper unipotent with the
top right entry $x\in\frakp/\frakp^{2m+2}$, we have 
\[
\alpha_{0}^{m+2}\begin{pmatrix}\begin{array}{cc}
1 & x\\
0 & 1
\end{array}\end{pmatrix}\alpha_{0}^{-m-1}=\begin{pmatrix}\begin{array}{cc}
\varpi_{F}^{-1} & x\varpi_{F}^{-2m-3}\\
0 & \varpi_{F}
\end{array}\end{pmatrix}\in\begin{cases}
K_{0}\alpha_{0}I_{S}(1) & \text{if }\,x\in\frakp^{2m+2}\\
K_{0}\alpha_{0}^{l^{\prime}(x)}I_{S}(1) & \text{if }\,x\in\frakp\setminus\frakp^{2m+2}
\end{cases},
\]
where $l^{\prime}(x)<-1$. Hence, we get 
\[
\tau(f_{m+1})(\alpha_{0}^{-m-2})=f_{-1}(\alpha_{0}^{-m-2}\alpha_{0}^{m+1}w_{0})=w_{0}\cdot f_{-1}(\alpha_{0})=w_{0}\cdot v_{\sigma_{\vec{r}}},
\]
that is to say, $c_{m+2}=1$ as required. This completes the proof. 
\end{proof}
Finally we are ready to show that $\tau$ satisfies the condition
(C4). Given the results we have established, the proof is a formal
argument essentially same as \cite[Lemma 4.4]{Xu_freeness}.
\begin{lem}
\label{lem: condition 3-top component always increases} --- For
$n\geq0$, let $f\in B_{n+1,\sigma_{\vec{r}}}$. If $\tau(f)\in B_{n+1,\sigma_{\vec{r}}}$
then $f\in B_{n,\sigma_{\vec{r}}}$.
\end{lem}

\begin{proof}
Let $M_{n+1,\sigma_{\vec{r}}}$ be the subspace of $B_{n+1,\sigma_{\vec{r}}}$
consisting of functions $f$ such that $\tau(f)\in B_{n+1,\sigma_{\vec{r}}}$.
So we have to show that $M_{n+1,\sigma_{\vec{r}}}\subset B_{n,\sigma_{\vec{r}}}$.

Suppose for contradiction that there exists some $f\in M_{n+1,\sigma_{\vec{r}}}\setminus B_{n,\sigma_{\vec{r}}}$.
Then $f\in B_{n+1,\sigma_{\vec{r}}}=B_{n,\sigma_{\vec{r}}}\oplus C_{n+1,\sigma_{\vec{r}}}$
and so we write $f=f^{\prime}+f^{\prime\prime}$ for some $f^{\prime}\in B_{n,\sigma_{\vec{r}}},\,f^{\prime\prime}\in C_{n+1,\sigma_{\vec{r}}}$.
Now note that $B_{n,\sigma_{\vec{r}}}\subset M_{n+1,\sigma_{\vec{r}}}$
since $B_{n,\sigma_{\vec{r}}}\subset B_{n+1,\sigma_{\vec{r}}}$ and
$\tau(B_{n,\sigma_{\vec{r}}})\subset B_{n+1,\sigma_{\vec{r}}}$ by
Lemma \ref{lem: action of tau on 2n circles}. Hence $f^{\prime\prime}=f-f^{\prime}\in M_{n+1,\sigma_{\vec{r}}}$.
But since $f\notin B_{n,\sigma_{\vec{r}}}$ we have $f^{\prime\prime}\neq0$.
Thus $C_{n+1,\sigma_{\vec{r}}}\cap M_{n+1,\sigma_{\vec{r}}}\neq0$. 

By Remark \ref{rem : I_S(1) invariants of the sphere R_n(sigma)}
we know that $C_{n+1,\sig}$ is $I_{S}(1)$-stable for $n\geq0$ and
so $B_{n+1,\sig}$ is $I_{S}(1)$-stable for $n\geq0$. Also since
$\tau$ is $G_{S}$-invariant $M_{n+1,\sig}$ is $I_{S}(1)$-stable.
But since $I_{S}(1)$ is a pro-$p$ group there exists some non-zero
vector $f^{*}\in C_{n+1,\sigma_{\vec{r}}}\cap M_{n+1,\sigma_{\vec{r}}}$
which is fixed by $I_{S}(1)$. Now $C_{n+1,\sigma_{\vec{r}}}:=R_{n+1}^{+}(\sigma_{\vec{r}})\oplus R_{n}^{-}(\sigma_{\vec{r}})$
and hence $C_{n+1,\sigma_{\vec{r}}}^{I_{S}(1)}=R_{n+1}^{+}(\sigma_{\vec{r}})^{I_{S}(1)}\oplus R_{n}^{-}(\sigma_{\vec{r}})^{I_{S}(1)}=\ffpbar\cdot f_{-(n+1)}\oplus\ffpbar\cdot f_{n+1}$
by Remark \ref{rem : I_S(1) invariants of the sphere R_n(sigma)}.
We write $f^{*}=\lambda f_{-(n+1)}+\mu f_{n+1}$ and using Lemma \ref{lem:tau(f_n)}(ii)
we obtain that $\tau(f^{*})$ is of the form 
\[
c_{-(n+1)}f_{-(n+1)}+d_{n+1}f_{n+1}+f_{-(n+2)}+f_{n+2}\,,\qquad\text{ where }c_{-(n+1)},d_{n+1}\text{ are scalars}.
\]
 Hence $\tau(f^{*})$ does not lie in $B_{n+1,\sig}$ because $f_{-(n+2)},f_{n+2}\in C_{n+2,\sig}$.
Thus we have a contradiction.
\end{proof}
We now complete the proof of our main theorem.
\begin{proof}[Proof of Theorem \ref{thm:freeness theorem}]
 To apply Lemma \ref{lem: technical lemma} we take $V:={\rm ind}_{K_{0}}^{G_{S}}(\sig)$
and $C_{n}:=C_{n,\sig}$ for each $n\geq0$. The operator $T:=\tau$.
By Remark \ref{rem: condition C1} and Lemmas \ref{lem: action of tau on 2n circles},
\ref{lem:tau injective on C_0}, \ref{lem: condition 3-top component always increases}
we have ensured that the conditions (C1)-(C4) are satisfied. Thus,
as in Remark \ref{rem: d=00003D1 remark} we can choose non-empty
subsets $A_{k}\subset C_{k}$ for each $k\geq0$ such that $\bigcup_{n\geq0}\bigsqcup_{i+j\leq n}\tau^{i}(A_{j})$
is a basis of ${\rm ind}_{K_{0}}^{G_{S}}(\sig)$. Consequently, the
set $\bigcup_{n\geq0}A_{n}$ forms a basis for the module action of
$\mathcal{H}(G_{S},K_{0},\sig)=\ffpbar[\tau]$ on ${\rm ind}_{K_{0}}^{G_{S}}(\sig)$.
\end{proof}
\begin{rem}
-- The freeness result provides a transparent and conceptual understanding
of the compactly induced representations of $\SL_{2}$ in the mod-$p$
setting. Beyond its intrinsic interest, the heart of the matter involves
the verification of the hypotheses laid out in the formal framework
of the Lemma \ref{lem: technical lemma}, and this idea may be adapted
to split reductive groups of higher ranks when the corresponding Hecke
action exhibits a graded or triangular behavior. Note that it is essential
to consider split groups as in this case the spherical Hecke algebra
is known to be commutative (see \cite[Corollary 1.3]{Herzig_satake}).
Such constructions may play a role in the study of supersingular representations.
\end{rem}

\end{document}